\documentclass[11pt,a4paper]{article}
\usepackage[T1]{fontenc}
\usepackage[utf8]{inputenc}
\usepackage{lmodern}
\usepackage[margin=25mm]{geometry}
\usepackage{amsmath,amssymb,amsthm}
\usepackage{microtype}
\usepackage[hidelinks]{hyperref}

\usepackage{tikz}

\newtheorem{theorem}{Theorem}[section]
\newtheorem{lemma}[theorem]{Lemma}
\newtheorem{corollary}[theorem]{Corollary}
\theoremstyle{definition}
\newtheorem{definition}[theorem]{Definition}
\theoremstyle{remark}
\newtheorem{remark}[theorem]{Remark}

\DeclareMathOperator{\conv}{conv}
\DeclareMathOperator{\diam}{diam}
\DeclareMathOperator{\dist}{dist}

\newcommand{\T}{\mathcal T}

\numberwithin{equation}{section}

\title{Largest-dihedral-angle bisection algorithm does not preserve mesh regularity for tetrahedral partitions}
\author{Sergey Korotov \and J\'er\^ome Michaud}
\date{September 16, 2026}

\begin{document}
\maketitle
\begin{center}
Department of Business and Mathematics, IEM, Mälardalen University, Västerås, Sweden \newline
emails: sergey.korotov@mdu.se, jerome.michaud@mdu.se 
\end{center}

\begin{abstract}
We present a fixed nondegenerate tetrahedron with a nested largest-dihedral-angle-bisection branch that collapses onto a unit edge.  The selected dihedral angles to bisect are uniquely largest ones at every bisection step.  Some face angles of the branch elements tend to zero, the corresponding diameter-to-inradius ratio diverges, and the diameters of the elements produced do not tend to zero, although every dihedral angle stays uniformly away from both $0$ and $\pi$ and all face and dihedral angles satisfy a uniform maximum-angle bound. The proof is based on an exact invariant-range argument.  This example shows that the largest-dihedral-angle bisection algorithm can produce degenerating tetrahedral partitions. 
%The same bad elements occur in conforming adaptive partitions.
\end{abstract}



\section{Introduction and main result}

Planar geometric refinement rules often have strong regularity properties, e.g. many mesh regularity results are available for longest-edge bisection/trisection algorithms applied to triangulations, see e.g. the review \cite{KPS} and references therein. Some recent results on the subject are also contained in \cite{KS}.  Similarily, the largest-angle bisection preserves a positive lower bound for descendant angles \cite{IKKL}; more generally, largest-angle $n$-section preserves the minimum and maximum angle conditions and forces the maximum diameter of level-wise descendants to zero for every fixed $n\ge2$ \cite{MKLA}.  The direct 3D analogue considered here is based on bisecting the largest \emph{internal dihedral angles} in tetrahedra.

The three-dimensional situation in the context of bisection-based strategies seems quite different from the two-dimensional case.  Korotov \cite{KorLEB} gave a two-step longest-edge-bisection recurrence whose shape parameter is halved; the resulting branch degenerates, with one dihedral tending to $0$ and another to $\pi$.  That construction uses a recurrent longest-edge tie.  Adiprasito, Kalmanovich, and Solomon \cite{AKS} subsequently formulated longest-edge bisection as a projective dynamical system and obtained, among other results, degenerating tetrahedral branches with a uniquely longest edge at every step.  
The counterexample presented in this paper concerns a different rule, see the definition below.  It is tie-free, loses both shape regularity and diameter decay, and degenerates through face angles even though its dihedral angles remain uniformly nonextreme.

\begin{definition}[Largest-dihedral-angle bisection]
Let $T=\conv\{X,Y,Z,W\}$ be a nondegenerate tetrahedron.  Select an edge $XY$ carrying a largest internal dihedral angle.  The internal angle-bisector plane through $XY$ meets the opposite edge $ZW$ at an interior point $E$ and replaces $T$ by two (sub)tetrahedra $XYZE$ and $XYEW$.  We call this operation largest-dihedral-angle bisection, or LAB.  
\end{definition}

In general, $E$ is not the midpoint of $ZW$, i.e. the largest-dihedral-angle bisection is different from the longest-edge bisection.

Write $h_T=\diam T$ and let $r_T$ be the inradius.  The tetrahedral minimum angle condition requires a positive uniform lower bound for all face and internal dihedral angles; it is equivalent to boundedness of $h_T/r_T$ \cite{BKK,KKK}.  Put
\[
 \Gamma(T)=\max\{\text{all face angles and all internal dihedral angles of }T\}.
\]
The maximum angle condition means then  the existence of $G<\pi$ such that $\Gamma(T)\le G$ uniformly \cite{Krizek}.

\begin{theorem}[Fixed-seed counterexample]\label{thm:main}
Let
\begin{equation}\label{eq:seed}
\begin{gathered}
 O=(0,0,0),\qquad P=(0,0,1),\\
 A_0=\left(\frac3{16},0,0\right),\qquad
 B_0=\left(\frac7{32},\frac{\sqrt{15}}{32},0\right),\qquad
 T_0=\conv\{O,P,A_0,B_0\}.
\end{gathered}
\end{equation}
There is a nested LAB branch $T_n=\conv\{O,P,A_n,B_n\}$ such that $PA_n$ carries the unique largest dihedral of $T_n$ for every $n\ge0$.  Moreover,
\begin{align}
 \angle OPA_n&\le \frac15\left(\frac45\right)^n, &
 \frac{h_{T_n}}{r_{T_n}}&\ge
 \frac{64}{\sqrt{15}}\left(\frac54\right)^n,\label{eq:rates-main}\\
 h_{T_n}&\longrightarrow1, &
 \bigcap_{n\ge0}T_n&=[O,P].\label{eq:limit-main}
\end{align}
Nevertheless, with
\[
 \gamma=\arccos\frac78,\qquad G=\arccos\left(-\frac14\right),
\]
every internal dihedral $\theta_e(T_n)$ and every face angle satisfy
\begin{equation}\label{eq:angle-main}
 \gamma\le\theta_e(T_n)\le G,\qquad \Gamma(T_n)\le G<\pi.
\end{equation}
Finally, for every $n$, the tetrahedron $T_n$ belongs to a conforming partition of $T_0$ obtained by $n$ LAB cuts and containing exactly $n+1$ tetrahedra.
\end{theorem}

Thus one fixed nondegenerate initial element, of volume $\sqrt{15}/1024$, produces a tie-free branch violating the minimum angle condition and diameter decay while satisfying the angular maximum angle condition. This situation is allustrated in Figure~\ref{fig:sequence}.

\section{An invariant family and its exact recurrence}

Set
\begin{equation}\label{eq:cs}
 c=\frac78,\qquad s=\frac{\sqrt{15}}8,
\end{equation}
and consider, up to a rigid motion fixing $O$ and $P$,
\begin{equation}\label{eq:family}
 T(a,b)=\conv\{O,P,A,B\},\qquad
 A=(a,0,0),\quad B=(bc,bs,0),
\end{equation}
where
\begin{equation}\label{eq:range}
 0<b\le\frac14,
 \qquad \frac34\le t:=\frac ab\le\frac45.
\end{equation}
The seed \eqref{eq:seed} corresponds to $a_0=3/16$ and $b_0=1/4$.  Define
\begin{equation}\label{eq:LW}
 L=\sqrt{a^2+b^2-2abc},\qquad
 W=\sqrt{a^2+b^2-2abc+a^2b^2s^2}.
\end{equation}
Here $L=|AB|$ and $W$ is twice the area of the face $PAB$.

\begin{lemma}[Unique selected edge]\label{lem:unique}
For every $T(a,b)$ satisfying \eqref{eq:range}, the edge $PA$ carries the unique largest internal dihedral angle.
\end{lemma}

\begin{proof}
Outward unit normals to $OAB$, $OPA$, $OPB$, and $PAB$ are, respectively,
\[
 (0,0,-1),\quad (0,-1,0),\quad (-s,c,0),\quad
 \frac{(bs,a-bc,abs)}{W}.
\]
If $u,v$ are adjacent outward normals, the corresponding internal dihedral satisfies $\cos\theta=-u\cdot v$.  Hence
\begin{equation}\label{eq:dihedrals}
\begin{aligned}
 \cos\theta_{PA}&=\frac{a-bc}{W},&
 \cos\theta_{PB}&=\frac{b-ac}{W},&
 \cos\theta_{OP}&=c,\\
 \theta_{OA}&=\theta_{OB}=\frac\pi2,&&
 \cos\theta_{AB}&=\frac{abs}{W}.
\end{aligned}
\end{equation}
Since $a/b\le4/5<7/8=c$, the dihedral at $PA$ is obtuse.  The other five are acute or right because $b-ac=b(1-tc)>0$.  Thus $PA$ is uniquely selected.
\end{proof}

\begin{lemma}[Spatial angle-bisector ratio]\label{lem:bisector}
If the internal dihedral bisector through $XY$ meets the opposite edge $ZW$ at $E$, then
\begin{equation}\label{eq:bisector}
 \frac{|ZE|}{|EW|}=
 \frac{\dist(Z,\operatorname{line}(XY))}
 {\dist(W,\operatorname{line}(XY))}.
\end{equation}
\end{lemma}

\begin{proof}
Project orthogonally onto a plane perpendicular to $XY$.  The edge $XY$ becomes a point, the two incident faces become rays, and the bisector plane becomes their ordinary internal angle bisector.  Formula \eqref{eq:bisector} is the planar angle-bisector theorem, since orthogonal projection is affine on $ZW$.
\end{proof}

Apply Lemma~\ref{lem:bisector} to the selected edge $PA$ of $T(a,b)$.  Its bisector meets $OB$ at a point $E$.  Since
\[
 \dist(O,PA)=\frac{a}{\sqrt{1+a^2}},\qquad
 \dist(B,PA)=\frac{W}{\sqrt{1+a^2}},
\]
we obtain, writing $e=|OE|$,
\begin{equation}\label{eq:e}
 \frac{e}{b-e}=\frac aW,
 \qquad e=\frac{ab}{a+W}.
\end{equation}
Retain the child $\conv\{O,P,A,E\}$ and relabel $A'=E$, $B'=A$.  The two base rays exchange roles, so the child is again of the form \eqref{eq:family}, with
\begin{equation}\label{eq:return}
 \boxed{\quad a'=\frac{ab}{a+W},\qquad b'=a.\quad}
\end{equation}

\begin{lemma}[Invariant range]\label{lem:invariant}
The recurrence \eqref{eq:return} maps the parameter region \eqref{eq:range} into itself.
\end{lemma}

\begin{proof}
With $a=tb$, write
\[
 Q=\sqrt{1+t^2-2ct+s^2t^2b^2}.
\]
Then
\begin{equation}\label{eq:F}
 t'=\frac{a'}{b'}=F(t,b):=\frac1{t+Q}.
\end{equation}
The function $F$ decreases with $b$.  It also decreases with $t$, because
\[
 \frac{\partial}{\partial t}(t+Q)
 =1+\frac{t-c+s^2tb^2}{Q}>0,
 \qquad Q^2=(t-c)^2+s^2(1+t^2b^2)>(t-c)^2.
\]
After continuously extending $F$ to $b=0$, its maximum on
$[3/4,4/5]\times[0,1/4]$ is
\[
 F\left(\frac34,0\right)=\frac45.
\]
At the opposite corner $Q^2=399/1600$, and therefore
\[
 F(t,b)\ge F\left(\frac45,\frac14\right)
 =\frac1{4/5+\sqrt{399}/40}>\frac{10}{13}>\frac34.
\]
Finally, $0<b'=a=tb\le(4/5)b\le1/4$.  Both restrictions in \eqref{eq:range} are preserved.
\end{proof}

Lemmas~\ref{lem:unique} and \ref{lem:invariant} prove that the recurrence can be continued indefinitely and that every selected LAB edge is unique.  This exact return mechanism is analogous in spirit to the contracting parameter recurrences used for degenerating longest-edge-bisection orbits \cite{KorLEB,AKS}, but here it is obtained directly from the dihedral bisector geometry.

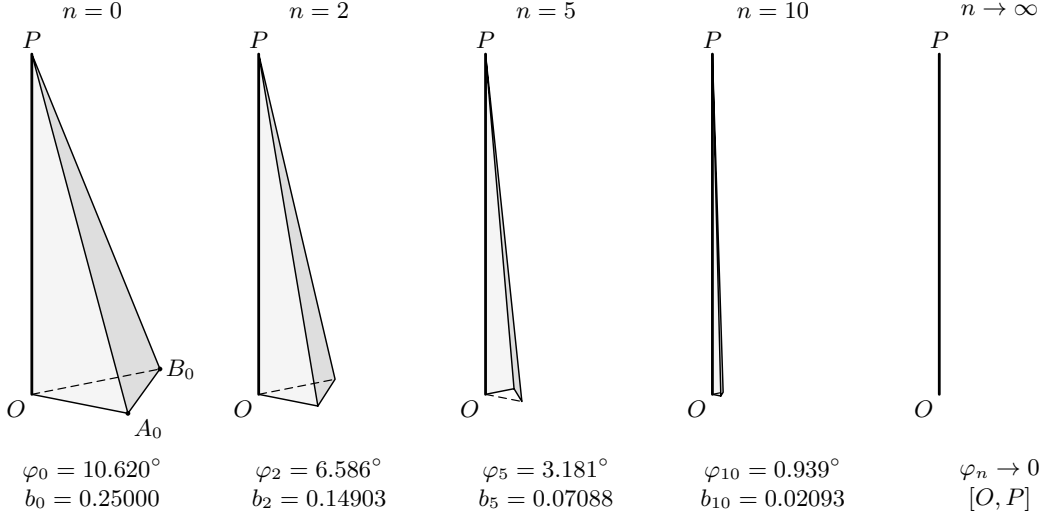
\begin{figure}[tbp]
\centering
\begin{tikzpicture}[font=\footnotesize,line join=round,line cap=round]
\path[use as bounding box] (-.15,-1.60) rectangle (15.25,5.28);
\begin{scope}[shift={(0.4200,0)}]
\coordinate (O) at (0.000000000,0.000000000);
\coordinate (P) at (0.000000000,4.499513268);
\coordinate (A) at (1.270822660,-0.251358333);
\coordinate (B) at (1.694430214,0.335144444);
\fill[black!4] (O)--(P)--(A)--cycle;
\fill[black!13] (P)--(A)--(B)--cycle;
\draw[line width=.55pt] (O)--(A)--(B)--(P)--(A);
\draw[densely dashed,line width=.45pt] (O)--(B);
\draw[line width=1.05pt] (O)--(P);
\node at (.8,5.08) {$n=0$};
\node at (.8,-1.00) {$\varphi_{0}=10.620^\circ$};
\node at (.8,-1.38) {$b_{0}=0.25000$};
\fill (A) circle (.9pt);\node[below right,inner sep=2pt] at (A) {$ A_0 $};
\fill (B) circle (.9pt);\node[right,inner sep=2pt] at (B) {$ B_0 $};
\node[above,inner sep=2pt] at (P) {$P$};\node[below left,inner sep=2pt] at (O) {$O$};
\end{scope}
\begin{scope}[shift={(3.4200,0)}]
\coordinate (O) at (0.000000000,0.000000000);
\coordinate (P) at (0.000000000,4.499513268);
\coordinate (A) at (0.782503006,-0.154772697);
\coordinate (B) at (1.010054032,0.199780430);
\fill[black!4] (O)--(P)--(A)--cycle;
\fill[black!13] (P)--(A)--(B)--cycle;
\draw[line width=.55pt] (O)--(A)--(B)--(P)--(A);
\draw[densely dashed,line width=.45pt] (O)--(B);
\draw[line width=1.05pt] (O)--(P);
\node at (.8,5.08) {$n=2$};
\node at (.8,-1.00) {$\varphi_{2}=6.586^\circ$};
\node at (.8,-1.38) {$b_{2}=0.14903$};
\node[above,inner sep=2pt] at (P) {$P$};\node[below left,inner sep=2pt] at (O) {$O$};
\end{scope}
\begin{scope}[shift={(6.4200,0)}]
\coordinate (O) at (0.000000000,0.000000000);
\coordinate (P) at (0.000000000,4.499513268);
\coordinate (A) at (0.376673569,0.074502953);
\coordinate (B) at (0.480433637,-0.095025846);
\fill[black!4] (O)--(P)--(A)--cycle;
\fill[black!13] (P)--(A)--(B)--cycle;
\draw[line width=.55pt] (O)--(A)--(B)--(P)--(A);
\draw[densely dashed,line width=.45pt] (O)--(B);
\draw[line width=1.05pt] (O)--(P);
\node at (.8,5.08) {$n=5$};
\node at (.8,-1.00) {$\varphi_{5}=3.181^\circ$};
\node at (.8,-1.38) {$b_{5}=0.07088$};
\node[above,inner sep=2pt] at (P) {$P$};\node[below left,inner sep=2pt] at (O) {$O$};
\end{scope}
\begin{scope}[shift={(9.4200,0)}]
\coordinate (O) at (0.000000000,0.000000000);
\coordinate (P) at (0.000000000,4.499513268);
\coordinate (A) at (0.111125846,-0.021979784);
\coordinate (B) at (0.141834299,0.028053665);
\fill[black!4] (O)--(P)--(A)--cycle;
\fill[black!13] (P)--(A)--(B)--cycle;
\draw[line width=.55pt] (O)--(A)--(B)--(P)--(A);
\draw[densely dashed,line width=.45pt] (O)--(B);
\draw[line width=1.05pt] (O)--(P);
\node at (.8,5.08) {$n=10$};
\node at (.8,-1.00) {$\varphi_{10}=0.939^\circ$};
\node at (.8,-1.38) {$b_{10}=0.02093$};
\node[above,inner sep=2pt] at (P) {$P$};\node[below left,inner sep=2pt] at (O) {$O$};
\end{scope}
\begin{scope}[shift={(12.4200,0)}]
\coordinate (O) at (0.000000000,0.000000000);
\coordinate (P) at (0.000000000,4.499513268);
\draw[line width=1.05pt] (O)--(P);
\node at (.8,5.08) {$n\to\infty$};
\node[above,inner sep=2pt] at (P) {$P$};
\node[below left,inner sep=2pt] at (O) {$O$};
\node at (.8,-1.00) {$\varphi_n\to0$};
\node at (.8,-1.38) {$[O,P]$};
\end{scope}
\end{tikzpicture}
\caption{The actual descendants $T_0,T_2,T_5,T_{10}$ and their limiting segment. Every panel uses the same orthographic projection and the same spatial scale; there is no per-panel rescaling or rotation. The bold edge $OP$ has length one throughout. Here $b_n=|OB_n|$ and $\varphi_n=\angle OPA_n$. The lateral vertices exchange their labels after each cut, but both approach $O$.}
\label{fig:sequence}
\end{figure}

\section{Degeneration and angle bounds}

\begin{theorem}[Three geometric conclusions]\label{thm:geometry}
The LAB branch generated from \eqref{eq:seed} satisfies \eqref{eq:rates-main}--\eqref{eq:angle-main}.
\end{theorem}

\begin{proof}
Let $(a_n,b_n)$ be its parameters and $t_n=a_n/b_n$.  From $b_{n+1}=a_n=t_nb_n$ and \eqref{eq:range},
\begin{equation}\label{eq:ab-bounds}
 \frac14\left(\frac34\right)^n\le b_n
 \le\frac14\left(\frac45\right)^n,
 \qquad
 a_n=t_nb_n\le\frac15\left(\frac45\right)^n.
\end{equation}
The face $OPA_n$ is right-angled at $O$, with $OP=1$, so
\[
 \angle OPA_n=\arctan a_n
 \le\frac15\left(\frac45\right)^n\longrightarrow0.
\]

The tetrahedron lies in a slab of width $b_ns$ bounded by the plane $OPA_n$ and its parallel translate through $B_n$.  Any inscribed ball therefore has diameter at most $b_ns$.  Since $OP\subset T_n$,
\[
 r_{T_n}\le\frac{b_ns}{2},\qquad h_{T_n}\ge1,
\]
and the upper estimate for $b_n$ gives
\[
 \frac{h_{T_n}}{r_{T_n}}
 \ge\frac2{sb_n}
 \ge\frac{64}{\sqrt{15}}\left(\frac54\right)^n.
\]
Because $b_n\ge a_n$ and all base edges are shorter than one, the longest edge is $PB_n$.  Hence
\[
 h_{T_n}=\sqrt{1+b_n^2}\longrightarrow1.
\]
Both base vertices approach $O$, whereas $P$ stays fixed.  The tetrahedra are nested and contain $OP$, so their intersection is exactly $[O,P]$.

It remains to control the angles from above and the dihedrals from below.  Let $\alpha(t)=\angle OAB$.  The cosine rule gives
\begin{equation}\label{eq:alpha}
 \cos\alpha(t)=\frac{t-c}{\sqrt{1+t^2-2ct}}.
\end{equation}
The derivative of the right-hand side is
$s^2(1+t^2-2ct)^{-3/2}>0$.  Therefore $\alpha$ decreases on the invariant interval and
\begin{equation}\label{eq:alpha-bound}
 \frac\pi2<\alpha(t)\le\alpha\left(\frac34\right)
 =\arccos\left(-\frac14\right)=G.
\end{equation}
In \eqref{eq:dihedrals}, $a-bc<0$ and $W\ge L$, so
$\theta_{PA}\le\alpha(t)\le G$.  All other dihedrals are at most $\pi/2$.

For the lower bound, $\theta_{OP}=\gamma$ and
$\theta_{OA}=\theta_{OB}=\pi/2$.  Moreover,
\[
 \cos\theta_{PB}
 \le\frac{1-ct}{\sqrt{1+t^2-2ct}}
 \le\frac{11}{16}<c,
\]
where the middle quotient decreases with $t$.  From \eqref{eq:F}, $Q\ge s$, and hence
\[
 \cos\theta_{AB}=\frac{as}{Q}\le a\le\frac15<c.
\]
The remaining angle $\theta_{PA}$ is obtuse.  Thus every dihedral is at least $\gamma$.

Finally, the base $OAB$ has largest angle $\alpha(t)\le G$, while $OPA$ and $OPB$ are right triangles.  In $PAB$, the angles at $B$ and $P$ are acute because the relevant scalar products are $b(b-ac)>0$ and $1+abc>0$.  At $A$,
\[
 \cos\angle PAB=
 \frac{a}{\sqrt{1+a^2}}\cos\alpha(t).
\]
Since $\cos\alpha(t)<0$ and the prefactor lies in $(0,1)$,
$\angle PAB<\alpha(t)\le G$.  This proves \eqref{eq:angle-main}.
\end{proof}

\begin{remark}
The contrast with the explicit longest-edge example of \cite{KorLEB} is useful.  There degeneration is accompanied by dihedrals tending to both $0$ and $\pi$, so both angle conditions fail.  Here no dihedral degenerates and the angular maximum-angle condition holds; the loss of regularity occurs through a face angle tending to zero.  The example therefore separates dihedral-angle control from the full tetrahedral minimum-angle condition.
\end{remark}

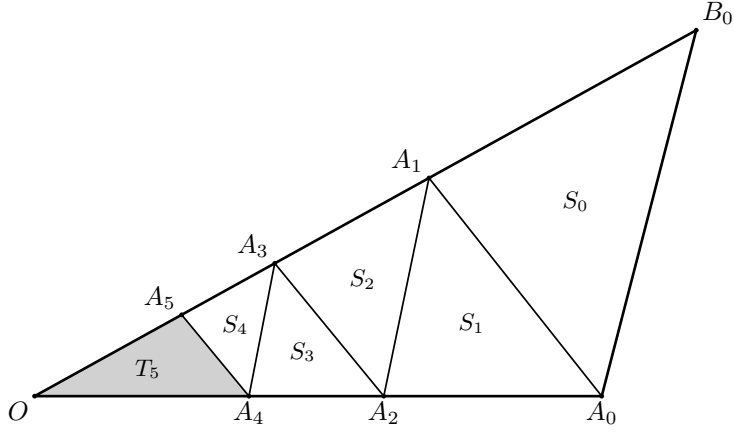
\begin{figure}[tbp]
\centering
\begin{tikzpicture}[font=\small,line join=round,line cap=round]
\path[use as bounding box] (-.4,-.7) rectangle (10.25,5.4);
\coordinate (O) at (0.000000000,0.000000000);
\coordinate (B0) at (8.750000000,4.841229183);
\coordinate (A0) at (7.500000000,0.000000000);
\coordinate (A1) at (5.215896595,2.885868664);
\coordinate (A2) at (4.618089311,0.000000000);
\coordinate (A3) at (3.176091697,1.757278607);
\coordinate (A4) at (2.835369867,0.000000000);
\coordinate (A5) at (1.945133946,1.076210197);
\fill[black!18] (O)--(A5)--(A4)--cycle;
\draw[line width=1pt] (O)--(A0)--(B0)--cycle;
\draw[line width=.65pt] (A0)--(A1);
\node[font=\footnotesize,inner sep=1pt] at (7.15529886,2.57569928) {$S_0$};
\draw[line width=.65pt] (A1)--(A2);
\node[font=\footnotesize,inner sep=1pt] at (5.77799530,0.96195622) {$S_1$};
\draw[line width=.65pt] (A2)--(A3);
\node[font=\footnotesize,inner sep=1pt] at (4.33669253,1.54771576) {$S_2$};
\draw[line width=.65pt] (A3)--(A4);
\node[font=\footnotesize,inner sep=1pt] at (3.54318363,0.58575954) {$S_3$};
\draw[line width=.65pt] (A4)--(A5);
\node[font=\footnotesize,inner sep=1pt] at (2.65219850,0.94449627) {$S_4$};
\node[font=\footnotesize] at (1.5,.36) {$T_5$};
\fill (O) circle (.9pt);\node[below left,inner sep=2pt] at (O) {$ O $};
\fill (B0) circle (.9pt);\node[above right,inner sep=2pt] at (B0) {$ B_0 $};
\fill (A0) circle (.9pt);\node[below,inner sep=2pt] at (A0) {$ A_0 $};
\fill (A1) circle (.9pt);\node[above left,inner sep=2pt] at (A1) {$ A_1 $};
\fill (A2) circle (.9pt);\node[below,inner sep=2pt] at (A2) {$ A_2 $};
\fill (A3) circle (.9pt);\node[above left,inner sep=2pt] at (A3) {$ A_3 $};
\fill (A4) circle (.9pt);\node[below,inner sep=2pt] at (A4) {$ A_4 $};
\fill (A5) circle (.9pt);\node[above left,inner sep=2pt] at (A5) {$ A_5 $};
\end{tikzpicture}
\caption{A planar view of the conforming adaptive construction after five cuts. Coning each triangle to the fixed apex $P$ gives a tetrahedron. The shaded triangle is the base of $T_5$; the five remaining triangles are the bases of the unrefined siblings $S_0,\ldots,S_4$. Every new point lies on one of the two original boundary edges $OA_0$ and $OB_0$, so no hanging node is created on an interior face. Note that $B_5=A_4$.}
\label{fig:partition}
\end{figure}

\section{Conforming adaptive partitions}

\begin{theorem}[Conforming realization]\label{thm:conforming}
For each $n\ge0$, there is a conforming tetrahedral partition $\T_n$ of $T_0$, obtained by $n$ LAB cuts, such that $T_n\in\T_n$ and $\#\T_n=n+1$.
\end{theorem}

\begin{proof}
At step $n$, refine only $T_n$ and write its two children as
\begin{equation}\label{eq:children}
 S_n=\conv\{P,A_n,B_n,A_{n+1}\},\qquad
 T_{n+1}=\conv\{O,P,A_{n+1},B_{n+1}\},
\end{equation}
where $B_{n+1}=A_n$.  Set
\[
 \T_n=\{S_0,\ldots,S_{n-1},T_n\},\qquad \T_0=\{T_0\}.
\]
In the base plane $z=0$, the active triangle is $OA_nB_n$, and the cut joins $A_n$ to a point $A_{n+1}$ on $OB_n$.  The segment $OB_n$ lies on one of the original boundary edges $OA_0$ and $OB_0$.  Thus every step subdivides only a boundary edge of the current planar triangulation and leaves the previous interior edge $A_nB_n$ intact.  The base triangulation is conforming.

Coning every base triangle to the common apex $P$ gives a face-to-face tetrahedral partition.  Each cut replaces one element by two, so $\#\T_n=n+1$.  Lemma~\ref{lem:unique} shows that every cut is a genuine, tie-free LAB cut. See Figure~\ref{fig:partition} for an illustration.
\end{proof}

\begin{corollary}[Complete local tree]\label{cor:tree}
Let $\mathcal D_n(T_0)$ be the depth-$n$ local LAB descendants of the seed \eqref{eq:seed}.  Independently of tie-breaking outside the constructed branch,
\[
 \min_{T\in\mathcal D_n(T_0)}\min\{\text{face angles of }T\}\longrightarrow0,
 \qquad
 \max_{T\in\mathcal D_n(T_0)}h_T\ge h_{T_n}\ge1.
\]
Thus the complete local process guarantees neither mesh regularity nor decay of its maximum element diameter.
\end{corollary}

\begin{proof}
The branch member $T_n$ occurs at depth $n$, and Theorem~\ref{thm:geometry} applies to it.
\end{proof}

The level-wise local tree in Corollary~\ref{cor:tree} is not asserted to be conforming.  Conformity is instead supplied by Theorem~\ref{thm:conforming}; global closure and alternative marking rules are separate algorithms.

\section{Conclusion}

Largest-angle refinement in the plane controls the complete angle geometry and, for largest-angle $n$-section, also yields diameter decay \cite{IKKL,MKLA}.  Neither conclusion transfers to the tetrahedral LAB rule.  The fixed seed \eqref{eq:seed} has a uniquely selected infinite branch whose face angles degenerate and whose diameter tends to one, even though all dihedrals remain in
\[
 \left[\arccos\frac78,\ \arccos\left(-\frac14\right)\right].
\]
Every face and dihedral angle is also bounded above by $G<\pi$.

Together with the longest-edge counterexamples \cite{KorLEB,AKS}, this shows that natural planar selection principles do not pass directly to tetrahedra.  Longest-edge bisection can retain diameter convergence while losing shape regularity, whereas the LAB branch constructed here loses both.  The conclusion is deliberately specific: it does not concern largest-face-angle rules, midpoint rules, combined safeguards, or the maximum-angle behaviour of every branch from every initial tetrahedron.


\begin{thebibliography}{99}
\small

\bibitem{AKS}
K.~A.~Adiprasito, D.~Kalmanovich, and Y.~Solomon,
\emph{Degenerating orbits of the longest edge bisection process},
arXiv:2609.08846 (2026).

\bibitem{BKK}
J.~Brandts, S.~Korotov, and M.~K\v{r}\'i\v{z}ek,
\emph{On the equivalence of regularity criteria for triangular and tetrahedral finite element partitions},
Comput. Math. Appl. \textbf{55} (2008), 2227--2233.

\bibitem{IKKL}
D.~Ismailescu, J.~Kim, K.~Kim, and J.~Lee,
\emph{The largest angle bisection procedure},
arXiv:1908.02749v2 (2019).

%\bibitem{IKST}
%H.~Ishizaka, K.~Kobayashi, %R.~Suzuki, and T.~Tsuchiya,
%\emph{A new geometric condition %equivalent to the maximum angle %condition for tetrahedrons},
%Comput. Math. Appl. \textbf{99} %(2021), 323--328.

\bibitem{KS}
D.~Kalmanovich and Y.~Solomon,
\emph{On the stability, complexity, and distribution of similarity classes of the longest edge bisection process for triangles},
arXiv:2601.13663  (2026).

\bibitem{KorLEB}
S.~Korotov,
\emph{The longest-edge bisection algorithm may produce degenerating tetrahedra},
arXiv:2608.23139 (2026).


\bibitem{KKK}
S.~Korotov, M.~K\v{r}\'i\v{z}ek, and V.~Ku\v{c}era,
\emph{On degenerating finite element tetrahedral partitions},
Numer. Math. \textbf{152} (2022), 307--329.



\bibitem{KPS}
S.~Korotov, \'A. Plaza, J. P. Suarez,
\emph{Longest-edge $n$-section algorithms: properties and open problems},
J. Comput. Appl. Math. \textbf{293} (2016), 139--146.




\bibitem{Krizek}
M.~K\v{r}\'i\v{z}ek,
\emph{On the maximum angle condition for linear tetrahedral elements},
SIAM J. Numer. Anal. \textbf{29} (1992), 513--520.

\bibitem{MKLA}
J.~Michaud and S.~Korotov,
\emph{On triangulations generated by the largest-angle $n$-section algorithm},
arXiv:2607.25457 (2026).

\end{thebibliography}
\end{document}